\documentclass[10pt]{amsart}
\usepackage{amsfonts}
\usepackage{amssymb}
\usepackage{amsmath, amsthm, latexsym, amscd}
\usepackage[all]{xy}
\usepackage{hyperref}
\usepackage{subcaption}
\usepackage[margin=1.5in]{geometry}
\usepackage{enumerate}
\usepackage{tikz}\usetikzlibrary{matrix}\usetikzlibrary{arrows}
\usepackage{xcolor}
\usepackage{tikz-cd}
\tikzset{%
    symbol/.style={%
        draw=none,
        every to/.append style={%
            edge node={node [sloped, allow upside down, auto=false]{$#1$}}}
    }
}

\def \C{\textup{C}}
\def \CC{\textup{\r{C}}}
\def \Z{\mathbb{Z}}
\def \R{\mathbb{R}}
\def \Q{\mathbb{Q}}

\def \F{\mathcal{F}}

\def \O{\mathcal{O}}
\def \W{\mathcal{W}}
\def \X{\mathcal{X}}
\def \I{\mathcal{I}}
\def \J{\mathcal{J}}

\def \k{{\bf k}}

\def \B{\mathbb{B}}

\def\v{\mathfrak{v}}
\def\w{\mathfrak{w}}

\def \Trop{\operatorname{Trop}}

\def \In{\operatorname{in}}
\def \ord{\operatorname{ord}}
\def \Proj{\operatorname{Proj}}
\def \gr{\operatorname{gr}}
\def \Spec{\operatorname{Spec}}

\def \GL{\operatorname{GL}}
\def \M{\operatorname{M}}

\def \add{\oplus}
\def \ADD{\bigoplus}
\def \scale{\odot}
\def\bx{{\bf x}}

\theoremstyle{plain}
\newtheorem{theorem}{Theorem}[section]
\newtheorem{lemma}[theorem]{Lemma}
\newtheorem{proposition}[theorem]{Proposition}
\newtheorem{corollary}[theorem]{Corollary}

\theoremstyle{definition}
\newtheorem{example}[theorem]{Example}
\newtheorem{definition}[theorem]{Definition}
\newtheorem{remark}[theorem]{Remark}

\begin{document}
\title{Generic tropical initial ideals of Cohen-Macaulay algebras}

\author{Kiumars Kaveh}
\address{Department of Mathematics, University of Pittsburgh,
Pittsburgh, PA, USA}
\email{kaveh@pitt.edu}

\author{Christopher Manon}
\address{Department of Mathematics, University of Kentucky, Lexington, KY, USA}
\email{Christopher.Manon@uky.edu}

\author{Takuya Murata}
\address{Department of Mathematics, University of Pittsburgh,
Pittsburgh, PA, USA}
\email{takusi@gmail.com}

\date{\today}

\thanks{The first author is partially supported by a National Science Foundation Grant (Grant ID: DMS-1601303) and Simons Collaboration Grant for Mathematicians.}

\thanks{The second author is partially supported by a Simons Collaboration Grant for Mathematicians.}


\begin{abstract}
We study the generic tropical initial ideals of a positively graded Cohen-Macaulay algebra $R$ over an algebraically closed field $\k$.   Building on work of R\"omer and Schmitz, we give a formula for each initial ideal, and we express the associated quasivaluations in terms of certain $I$-adic filtrations.  As a corollary, we show that in the case that $R$ is a domain, every initial ideal coming from the codimension $1$ skeleton of the tropical variety is prime, so ``generic presentations of Cohen-Macaulay domains are well-poised in codimension $1$."   
\end{abstract}

\date{\today}
\maketitle

\tableofcontents

\section{Introduction}
Let $I$ be a homogeneous prime ideal in a polynomial  ring $\k[{\bf x}]:=\k[x_1, \ldots, x_n]$, where $\k$ is a field. Given a vector $w \in \R^n$, it is well-known in Gr\"obner theory that the initial ideal $\In_w(I)$ can be considered as a deformation of $I$. Geometrically, one has a flat family over the affine line, or a \emph{flat degeneration}, whose generic fiber is the scheme $X=\Proj(\k[{\bf x}] / I)$ and the special fiber (fiber over $0$) is the scheme $X_0 = \Proj(\k[{\bf x}] / \In_w(I))$ (see \cite[Section 15.8]{E}). The case when $\In_w(I)$ is prime is particularly interesting. In this case all the fibers of the family are reduced and irreducible. If we moreover assume that $w$ comes from relative interior of a maximal cone in the tropical variety of $I$ then $X_0$ is in fact a toric variety and one obtains a \emph{toric degeneration} of the projective variety $X$ (see \cite{Kaveh-Manon-NOK}). This is a very special and desirable situation and does not usually happen. 

In this paper we study the initial ideals of the \emph{generic tropical variety} of a homogeneous ideal $I$ in a polynomial ring. We show that when the ring is Cohen-Macaulay and the initial ideal comes from \emph{any} codimension $1$ cone in the generic tropical variety then the corresponding initial ideal is prime. Hence, in this generic situation, one has many ``almost toric'' degenerations of the projective variety $X$. This work can be considered as a follow-up to \cite{KMM}.

Let $R$ be a positively graded algebra over an algebraically closed field $\k$, and let $X = \Proj(R)$.   The theory of Newton-Okounkov bodies \cite{KK, LM} employs a $\Z^n$-valued valuation $\v$ on $R$, where $n=\dim(X)$, to construct a convex body $\Delta(R, \v)$ which plays the role for $X$ much like the Newton polytope (also called the  moment polytope) of a toric variety.  In general, the Newton-Okounkov body $\Delta(R, \v)$ can be poorly behaved.  It can be the case that $\Delta(R, \v)$ is not a polytope, and moreover, the value semigroup $S(R, \v) := \v(R \setminus \{0\})$, can fail to be finitely generated.  However, in the case that $S(R, \v)$ is finitely generated all is well: $\Delta(R, \v)$ is a polytope, and a theorem of Anderson \cite{Anderson} states that there is a toric degeneration of $X$ to a toric variety whose normalization is the projective toric variety associated to $\Delta(R, \v)$. In general, it is difficult to check the finite generation of $S(R, \v)$ and generically it is not expected to happen (see for example \cite{Anderson, HK, Ilten-Wrobel, Bossinger-toric-ideal}).

There has been much recent interest in finding toric degenerations of algebraic varieties, especially in the case where the variety has some interaction with representation theory (see \cite{Fang-Fourier-Littelmann} and references therein), where the combinatorics of $\Delta(R, \v)$ can be used to address counting problems. In \cite{Kaveh-Manon-NOK}, the first two authors show that toric degenerations can be detected and constructed using features of the \emph{tropical geometry} of $X$.  In particular, valuations which give rise to toric degeneration correspond to \emph{prime cones} in the tropical variety of $X$.  In recent work \cite{Escobar-Harada}, Escobar and Harada show that the Newton-Okounkov bodies associated to prime cones which share a face carry piecewise-linear \emph{mutation maps} between them. 

The ideal situation therefore is when all cones in the tropical variety are prime, in this case one says the variety is \emph{well-poised}.  The well-poised property was studied by Ilten and the second author \cite{Ilten-Manon} in the case of so-called rational $T$-varieties of complexity $1$. Such varieties carry a torus action of dimension one less than the dimension of the variety, and can therefore be thought of as being very close to toric varieties. 

Finally, in recent work \cite{KMM} the authors show that \emph{any} projective variety can be  degenerated to a (possible non-normal) $T$-variety of complexity $1$ using Bertini type techniques.  In this paper we combine the approaches of \cite{KMM} and \cite{Kaveh-Manon-NOK} to draw conclusions about the generic tropical variety of a Cohen-Macaulay variety.  In particular, building on work of R\"omer and Schmitz \cite{RS1, RS2}, the main result of the present paper  shows that every cone of codimension one in the generic tropical variety is prime, thus showing that such varieties are almost well-poised. An example of this situation is the Pl\"ucker algebra of the Grassmannian of 
$2$-planes. This is in fact an example of a well-poised variety. 

We begin by reviewing some background material. Let $\k$ be an algebraically closed field and let $\k[\bx] = \k[x_1, \ldots, x_n]$ be the polynomial ring in $n$ indeterminates. Take $f(\bx) = \sum_\alpha c_\alpha \bx^\alpha \in \k[\bx]$. Recall that for $w=(w_1, \ldots, w_n) \in \R^n$, the initial form $\In_w(f)$ is the polynomial $\sum_\beta c_\beta \bx^\beta$ where the sum is over all $\beta$ such that the inner product $\langle w, \beta \rangle$ is minimum. For an ideal $J \subset \k[\bx]$, $\In_w(J)$ is the ideal generated by $\In_w(f)$, for all $f \in J$. We also recall that the the tropical variety $\Trop(J) \subset \R^n$ is defined as:
$$\Trop(J) = \{ w \in \R^n \mid \In_w(J) \textup{ contains no monomials}\}.$$
It is well-known that $\Trop(J)$ is the support of a polyhedral fan (see \cite[Theorem 3.3.5]{MSt}). 

Fix a homogeneous ideal $J \subset \k[\bx]$ and consider the ideal $I = g\circ J$ obtained by making a coordinate transformation $g \in \GL_n(\k)$.   In \cite{RS1} it is shown that for $g$ in a nonempty Zariski open, $\Trop(I)$ is independent of the choice of $g$. Hence, $\Trop(I)$ for $g$ generic is known as the \emph{generic tropical variety} of $J$. In this paper, when $R = \k[\bx] / J$ is Cohen-Macaulay, we describe the initial ideals $\In_w(I)$ for $w$ in the the generic tropical variety $\Trop(I)$. 

First, let us consider the case of a principal ideal.
Take a homogeneous polynomial $\sum_{\deg(\alpha) = r} c_\alpha \bx^\alpha = f \in \k[\bx]$ with $c_\alpha \neq 0$, for all $\alpha$, that is, every monomial of degree $r$ appears in $f$ with nonzero coefficient. For an index set $A \subseteq [n] := \{1, \ldots, n\}$ we let $\C_A \subset \Q^n$ be the polyhedral cone of tuples $w = (w_1, \ldots, w_n)$ with the property that $w_i = \min\{w_j \mid 1 \leq j \leq n\}$ for all $i \in A^c := [n] \setminus A$, and $\CC_A \subset \C_A$ its relative interior.  It is straightforward to show that, for any $w \in \CC_A$, the initial form $\In_w(f)$ is the sum of those monomial terms $c_\beta \bx^\beta$ such that support of $\beta$ is a subset of $A^c$. Thus, $\In_w(f)$ does not contain any indeterminate $x_i$, $i \in A$. It follows that the quotient algebra $\k[\bx]/\langle \In_w(f) \rangle$ is a polynomial ring in $|A|$ variables over the algebra $\k[x_i \mid i \in A^c]/\langle \In_w(f) \rangle$.   If $|A^c| > 2$ and the coefficients $c_\alpha$ are chosen sufficiently generically, then both $f$ and $\In_w(f)$ are irreducible, so both $\k[\bx]/\langle f \rangle$ and its initial degeneration $\k[\bx]/\langle \In_w(f) \rangle$ are domains.   We generalize these observations to any ideal $J \subset \k[\bx]$ such that $\k[\bx]/J$ is Cohen-Macaulay. 

Now fix a homogeneous ideal $J \subset \k[\bx]$ and consider the ideal $I = g\circ J$ for generic $g \in \GL_n(\k)$.   In \cite{RS1} it is shown that if $\k[\bx]/I$ has Krull dimension $d$, the tropical variety $\Trop(I) \subseteq \Q^n$ is the support of the polyhedral fan $\W^n_d$ composed of the cones $\C_A$ with $|A^c| = n -d + 1$.  Moreover, the fan structure on $\Trop(I)$ detects whether $\k[\bx]/I$ is Cohen-Macaulay or a \emph{almost Cohen-Macaulay} (an algebra for which depth is dimension $-1$).  In particular, the quotient ring of $I$ is Cohen-Macaulay or almost Cohen-Macaulay if and only if the fan structure on $\W_d^n$ induced from the Gr\"obner fan of $I$ coincides with the fan structure defined by the cones $\C_A$ (see \cite[Corollary 4.7]{RS2}). 

Each point $w \in \Trop(I)$ can be used to construct a discrete homogeneous quasivaluation $\v_w: R\setminus \{0\} \to \Q$ and an associated graded algebra $\gr_w(R)$ (see Section \ref{filtration} as well as \cite[Section 2.4]{Kaveh-Manon-NOK}).  Here, $\gr_w(R) \cong \k[x_1, \ldots, x_n]/\In_w(I)$, where $\In_w(I)$ is the initial ideal of $I$ with respect to $w$.  In the following let $I = g \circ J$ for $g \in \GL_n(\k)$ chosen from an appropriate Zariski open subset, and let $y_i = \pi(x_i) \in \k[\bx]/I \cong R$, where $\pi$ is the quotient map.  Our first theorem sharpens the picture provided in \cite{RS2} by giving an explicit description of $\gr_w(R)$. Each $y_i$ gives a $\langle y_i \rangle$-adic filtration $\langle y_i \rangle \supset \langle y_i^2 \rangle \supset \cdots$ which in turn gives a quasivaluation $\ord_i$ on $R$ (see Section \ref{filtration}). We compute $\v_w$ in terms of the functions $\ord_i$ for $1 \leq i \leq n$.  We write $\min(w)$ as a shorthand for $\min\{ w_i \mid 1 \leq i \leq n\}$. 

\begin{theorem}\label{main-associatedgraded-quasival}
With notation as above, let $R$ be Cohen-Macaulay. Let $A \subset [n]$ with $|A^c| \leq n -d+1$ and $w \in \C_A$, then:  
\begin{enumerate}
\item $\gr_w(R) \cong (R/\langle y_i \mid i \in A \rangle)[t_i \mid i \in A]$,
\item $\v_w = (\min(w) \scale \deg) \add (\ADD_{i \in A} ((w_i- \min(w))\scale \ord_i)$.
\end{enumerate}
Here $\deg$ denotes the quasivaluation on $R$ given by homogeneous degree, and the operations $\add$ and $\scale$ are described in Section \ref{filtration}.  
\end{theorem}
As a corollary of Theorem \ref{main-associatedgraded-quasival} we obtain a description of each initial ideal $\In_w(I)$.  For a subset $A \subseteq [n]$ let $I_A \subset \k[x_j \mid j \in A^c]$ be the kernel of the induced presentation $\pi_A: \k[x_j \mid j \in A^c] \to R/\langle \{y_i \mid i \in A\}\rangle$. 

\begin{corollary}\label{cor-initialideal}
Let $w \in \CC_A$ be as above, then $\In_w(I) = I_A\k[x_1, \ldots, x_n]$.
\end{corollary}

\noindent
In particular, Corollary \ref{cor-initialideal} recovers and refines \cite[Proposition 4.6]{RS2}. 

Let  $\bar{y} = \{y_1, \ldots, y_n\} \subset  R$. The image of the generating set $\bar{y}$ generates each associated graded algebra $\gr_w(R)$. Following \cite[Section 2]{Kaveh-Manon-NOK}, $\bar{y}$ is said to be a \emph{Khovanskii basis} for $(R, \v_w)$. In fact all $\Q$-quasivaluations with Khovanskii basis $\bar{y}$ are of the form $\v_w$ for some $w \in \Q^n$.  Such a quasivaluation is a valuation precisely when the initial ideal $\In_w(I)$ is prime. A relatively open cone $\CC \subset \Trop(I)$ with $\In_u(I) = \In_{u'}(I)$ for all $u, u' \in \CC$ is said to be a \emph{prime cone} if $\In_w(I)$ is a prime ideal for all $w \in \CC$ (see \cite{Kaveh-Manon-NOK}).  The next result gives the state of affairs for prime cones in a generic tropicalization $\Trop(I)$.  

\begin{theorem}\label{main-primecones}
Let $R$, $J$, and $I$ be as above.  The ideal $J$ is radical if and only if $\In_w(I)$ is radical for all $w \in \W_d^n$.  The ideal $J$ is prime if and only if $\In_w(I)$ is prime for all $w \in \W_{d-1}^n \subset \W_d^n$. 
\end{theorem}

\noindent
Let ${\bf 1}$ denote the all $1$'s vector. From \cite[Theorem 1]{Kaveh-Manon-NOK} and Theorem \ref{main-primecones} it follows that if $R$ is a domain then any linearly independent collection $\w = \{w_1, \ldots, w_{d-2}, {\bf 1} \} \subset \C_A$ with $|A| = d-2$ defines an integral, rank $d-1$ valuation $\v_\w: R\setminus \{0\} \to \Z^{d-1}$ with a finite Khovanskii basis.  The associated graded algebra $\gr_\w(R)$ can be computed with Theorem \ref{main-associatedgraded-quasival}.  Valuations of this type define flat degenerations of $\Proj(R)$ to \emph{complexity-$1$ $T$-varieties}, and are studied in \cite{KMM}.

When the initial ideals coming from points in the tropical variety $\Trop(I)$ are all prime, $I$ is said to be \emph{well-poised}. This property was defined in \cite{Ilten-Manon}, where it was shown that the semi-canonical embeddings of rational, complexity-$1$ $T$-varieties are always well-poised.  Other examples are the Pl\"ucker embeddings of Grassmannian varieties of $2$-planes and any monomial-free linear ideal.  Theorem \ref{main-primecones} shows that the generic tropicalization of a Cohen-Macaulay domain almost has this property. 

\begin{corollary}  \label{cor-well-poised-codim1}
Let $R$ be Cohen-Macaulay, with $I$ as above. Then $\Trop(I)$ is well-poised in codimension $1$. That is, if $w$ belongs to a codimension $1$ cone in $\Trop(I)$, for the fan structure coming from the Gr\"obner fan, then 
$\In_w(I)$ is prime. 
\end{corollary}

\begin{remark}  \label{rem-KMM}
For a positively graded domain $R$ (not necessarily Cohen-Macaulay) \cite{KMM} uses the Bertini irreducibility theorem to construct a valuation with corank $1$ whose associated graded algebra is finitely generated. The above corollary shows that, when $R$ is Cohen-Macaulay, in fact one can construct such a valuation using any codimension $1$ cone in the generic tropical variety. This makes a direct connection between constructions in \cite{Kaveh-Manon-NOK} and \cite{KMM}.
\end{remark}

With above assumptions, it turns out that the additional property of being well-poised is a strong requirement. 

\begin{corollary}\label{cor-wellpoised}
Let $R$ be Cohen-Macaulay, with $I$ as above. Then $I$ is well-poised if and only if it is a linear ideal.  In particular, in this case $R$ must be isomorphic to a polynomial ring over $\k$.  
\end{corollary}

\noindent \textbf{Acknowledgement:} We would like to thank Lara Bossinger, Milena Wrobel, Mateusz Micha\l ek and Bernd Sturmfels for useful correspondence. The problem of studying ``almost-toric degenerations'' arising from prime cones of codimension $1$ in a tropical variety was suggested by Bernd Sturmfels. We also thank anonymous referees for useful suggestions and comments that greatly improved the content and presentation of the paper.

\section{Filtrations and quasivaluations}\label{filtration}

Let $\Gamma$ be an ordered group, and let $\bar{\Gamma} = \Gamma \cup \infty$. Recall that a \emph{quasivaluation} $\v: R \to \bar\Gamma$ over $\k$ is a function satisfying the following axioms for all $f, g \in R$:

\begin{itemize}
\item $\v(fg) \geq \v(f) +\v(g)$,
\item $\v(f + g) \geq \min\{\v(f), \v(g)\}$,
\item $\v(Cf) = \v(f)$, for all $C \in \k\setminus \{0\}$.
\end{itemize}

\noindent
We say $\v$ is a \emph{valuation} if $\v(fg) = \v(f) + \v(g)$.  Let $S(R, \v) \subset \Gamma$ be the set of values of $\v$; this is a semigroup under the group operation in $\Gamma$ if $\v$ is a valuation, and is referred to as the \emph{value semigroup} of $\v$ in this case.   The \emph{kernel} of a quasivaluation is the vector space of elements which are sent to $\infty$.  We say a quasivaluation $\v$ is homogeneous with respect to the grading on $R$ if the value $\v(f)$ of an element $f \in R$ is always achieved on one of its homogeneous components. In the sequel we deal with groups $\Gamma \subseteq \Q$, and we assume that the subgroup $\langle S(R, \v) \rangle \subseteq \Gamma$ is discrete.  Moreover, we take quasivaluations to be homogeneous (when this makes sense), and we assume the kernel is $\{0\}$, i.e. $\v(f) = \infty$ if and only if $f = 0$. 

For each quasivaluation $\v$ we obtain an $\Gamma$-algebra filtration of $R$ given by the spaces $$F_r(\v) = \{f \mid \v(f) \geq r\}.$$ In particular, for $r \leq s$ we have  $F_r(\v) \supseteq F_s(\v)$ and $F_r(\v)F_s(\v) \subseteq F_{r + s}(\v)$.   The filtration $F(\v)$ is homogeneous in the sense that $F_r(\v) = \bigoplus_{n \geq 0} F_r(\v) \cap R_n$, {where $R_n$ denotes the $n$-graded piece of $R$}.  In particular, the set of homogeneous elements of degree $d$ with value greater than or equal to $r$ is always a finite dimensional vector space.  Under this operation, the kernel of $\v$ is the vector space $\bigcap_{r \in \Gamma} F_r(\v)$.  Conversely, if we start with a homogeneous algebra filtration $F$ of $R$ with $\bigcap_{r \in \Gamma} F_r = \{0\}$, we obtain a quasivaluation $\v_F: R \to \bar{\Gamma}$ with kernel $\{0\}$, where $$\v_F(f) = \max\{r \mid \v(f) \geq r\}.$$  For any algebra filtration $F$ the \emph{associated-graded} algebra is defined as follows:

\begin{equation}
\gr_F(R) = \bigoplus_{r \in \Gamma} F_r/F_{> r}.
\end{equation}

\begin{example}
Any homogeneous ideal $I \subset R$ has a corresponding $I$-adic filtration by the powers $I^r \subset R$. We let $\v_I: R \to \bar{\Z}$ denote the quasivaluation corresponding to the $I$-adic filtration and $\gr_I(R)$ be the associated graded algebra. 
\end{example}

\begin{proposition}\label{prop-noker}
Let $\v:R \to \bar{\Gamma}$ be a quasivaluation with trivial kernel. If $\gr_\v(R)$ is reduced, then $R$ is reduced, and if $\gr_\v(R)$ is a domain, then $R$ is a domain. 
\end{proposition}
\begin{proof}
We show that if $R$ has a nilpotent or a zero divisor, then $\gr_\v(R)$ does as well.  Suppose $0 \neq f \in R$, and $\v(f) = r \in \Gamma$.  Let $\bar{f}$ be the image of $f \in F_r(\v)$ under the projection $F_r(\v) \to F_r(\v)/F_{> r}(\v)$.  By definition, $\bar{f}^n$ is computed by taking the image of $f^n \in F_{nr}(\v)$ under the projection $F_{nr}(\v) \to F_{nr}(\v)/F_{>nr}(\v)$.   If $f \neq 0$ and $f^n = 0$, then $\bar{f} \neq 0$, and $\bar{f}^n = 0$, since $f^n \in F_{>nr}(\v)$.   Similarly, if $f, g \neq 0$ with $\v(f) = r$, $\v(g) = s$ and $fg = 0$, then $\bar{f}, \bar{g} \neq 0$ and $\bar{f}\bar{g} = 0$ since $fg \in F_{>r + s}(\v)$.
\end{proof}

Now we define the operation $\add$ on the set of quasivaluations on $R$, this will be used in the proof of Theorem \ref{main-associatedgraded-quasival}.  

\begin{definition}\label{def-sum}
For quasivaluations $\v_1, \v_2$ on $R$, define $\v_1 \add \v_2$ to be the quasivaluation defined by the filtration composed of the following spaces:
\begin{equation}
F_r(\v_1 \add \v_2) = \sum_{r_1 + r_2 = r} F_{r_1}(\v_1) \cap F_{r_2}(\v_2).
\end{equation}
\end{definition}

Suppose $\Gamma$ is a $\Q$-vector space.  For $w \in \Q_{\geq 0}$ we let $w\scale \v$ denote the quasivaluation obtained by scaling the values of $\v$ by $w$, that is, $(w \scale \v)(f) = w\,\v(f)$, for all $f \in R$. It is straightforward to check that $F_r(w\scale \v) = F_{\frac{r}{w}}(\v)$ and that $n\scale \v = \ADD_{i =1}^n \v$ for $n \in \Z_{\geq 0}$.  From now on we will speak of $\add$ on filtrations and quasivaluations interchangeably.  

We say that a vector space basis $\mathbb{B}\subset R$ is an \emph{adapted basis} for a filtration $F$ if $F_r \cap \mathbb{B}$ is a basis for $F_r$, for all $r$.

The operation $\add$ is not associative in general.  However, $\add$ is associative on filtrations $F^1, F^2, \ldots, F^\ell$ if  the $F^i$ have a common   adapted basis $\mathbb{B}$. 

\begin{proposition}\label{prop-sum-properties}
Let $F^1, \ldots, F^\ell$ be homogeneous filtrations corresponding to discrete quasivaluations on $R$.   The subspaces $\{F^i_r\mid 1 \leq i \leq \ell, r \in \Gamma\}$ generate a distributive lattice of subspaces of $R$ (with operations of taking intersection and taking linear span of union) if and only if the $F^i$ share a common adapted basis. In this case, $\add$ defines an associative operation on the set of quasivaluations defined by taking multiples and $\add$-sums of $\v_{F^1}, \ldots, \v_{F^\ell}$. Finally, if $F$, $G$ are filtrations sharing a common adapted basis $\mathbb{B}$ then $\v_{F \add G}(b) = \v_F(b) + \v_G(b)$ for any $b \in \mathbb{B}$.    
\end{proposition}

\begin{proof}
If the $F^i$ share a common adapted basis $\mathbb{B}$, then each space $F^i_r$ corresponds to a subset $\B_r^i = \B \cap F_r^i$ and the operations of intersection and sum correspond to intersection and union of these subsets, respectively.  It follows that the $F_r^i$ generate a distributive lattice in the subspaces of $R$.   Conversely, if the $F_r^i$ generate such a lattice, then the same holds for their intersection with any graded component $R_n$. A distributive lattice of subspaces of a finite dimensional vector space always has an adapted basis (see \cite[Section 3]{DJS}).  Let $\mathbb{B}_n$ be this basis for $n \in \Z_{\geq 0}$, then $\mathbb{B} = \coprod_{n \geq 0} \mathbb{B}_n$ is adapted to each of the $F^i$.  

Now, for any $r \in \Q$ we have: 

\[ \sum_{s + s_3 = r} (\sum_{s_1 + s_2 = s} F^i_{s_1} \cap F^j_{s_2}) \cap F^k_{s_3} = \sum_{s_1 + s' = r} F^i_{s_1} \cap (\sum_{s_2 + s_3 = s'}  F^j_{s_2} \cap F^k_{s_3}),\]

\noindent
this implies that $((\v_{F^i}\add \v_{F^j}) \add \v_{F^k})^{-1}(r) =  (\v_{F^i}\add (\v_{F^j}) \add \v_{F^k}))^{-1}(r)$. Observe that any collection of filtrations obtained by $\add$ and scaling from the $F^i$ also share the basis $\mathbb{B}$, so the first part of this proof applies.  Finally, we let $b \in \mathbb{B} \subset R$ and we suppose that $F$, $G$ are filtrations adapted to $\mathbb{B}$.  If $b \in F_r$, $b \notin F_{< r}$ and $b \in G_s$, $b \notin G_{<s}$ then $b \in F_r \cap G_s \subset \sum_{t_1 + t_2 = r + s} F_{t_1} \cap G_{t_2}$. If $t_1 + t_2 < r + s$ then without loss of generality we may assume that $t_1 < r$.  It follows that $b \notin F_{t_1} \cap G_{t_2}$.  The set $\mathbb{B}$ is an adapted basis to each space $F_{t_1} \cap G_{t_2}$ and their sum, it follows that $b$ is in $\sum_{t_1 + t_2 < r + s} F_{t_1} \cap G_{t_2}$ if and only if $b$ is in one of the $F_{t_1} \cap G_{t_2}$, a contradiction.  
\end{proof}

Now we assume that $\Gamma = \Z$ or $\Q$ with the usual total ordering. An element $u \in \Gamma^n$ determines a valuation $\bar{\v}_u: \k[\bx] \to \bar{\Gamma}$ defined by sending a monomial $\bx^\alpha$ to the inner product $\langle u, \alpha \rangle$, and a polynomial $\sum c_\alpha \bx^\alpha$ to $\min\{\bar{\v}_u(\bx^\alpha) \mid c_\alpha \neq 0\}$.  A homogeneous presentation $\pi: \k[\bx] \to R$ then determines an associated \emph{weight quasivaluation} on $R$ by the pushforward operation: $\v_u = \pi_*\bar{\v}_u$.  In particular for $f \in R$ we have $\v_u(f) = \max\{\bar{\v}_u(p(\bx)) \mid \pi(p(\bx)) = f\}$ (see \cite[Definition 3.1]{Kaveh-Manon-NOK}).   

The associated graded algebra $\gr_u(R)$ of $\v_u$ is presented by the initial ideal $\In_u(I) \subset \k[\bx]$ (\cite[Lemma 3.4]{Kaveh-Manon-NOK}), where $I = \ker(\pi)$.   In particular, $\gr_u(R)$ is presented by the images of the generators $\pi(x_i) = b_i \in R$, $1 \leq i \leq n$.  This implies that the set $\mathcal{B} = \{b_1, \ldots, b_n\} \subset R_1 \subset R$ is a \emph{Khovanskii basis} of $\v_u$ (\cite[Definition 1]{Kaveh-Manon-NOK}). 

\begin{definition}[Khovanskii Basis]
Let $\v: A \to \Gamma$ be a quasivaluation. A set $\B \subset A$ is a {\it Khovanskii basis} if the image of $\B$ in the associated graded $\gr_\v(A)$ forms a set of algebra generators.
\end{definition}

\noindent
By \cite[Proposition 3.7]{Kaveh-Manon-NOK}, any quasivaluation with Khovanskii basis $\mathcal{B}$ is of the form $\v_u$ for some $u \in \Gamma^n$. 

Now we let $\Gamma = \Q$. The behavior of the weight quasivaluations $\v_u$ is governed by the \emph{Gr\"obner fan} $\Sigma(I)$. To a total monomial ordering $\prec$ we associate a closed cone $\tau_\prec \in \Sigma(I)$.  This is the set of $u$ such that $\In_\prec(\In_u(I) = \In_\prec(I)$.  The initial ideal $\In_\prec(I)$ is a monomial ideal; we let $\B_\prec \subset R$ be the set of images of monomials not contained in $\In_\prec(I)$.  It is well-known that $\B_\prec$ forms a vector space basis of $R$.  For the basics of Gr\"obner bases, and a proof of the following see \cite{GBCP} and \cite[Proposition 3.3]{Kaveh-Manon-NOK}.  

\begin{proposition}\label{prop-standard-properties}
Let $I \subset \k[x_1, \ldots, x_n]$ be a homogeneous ideal with Gr\"obner fan $\Sigma(I)$, then:
\begin{enumerate} 
	\item  For a monomial order $\prec$ and $u \in \tau_\prec$, the set $\B_\prec \subset R$ is an adapted basis of $\v_u$.
	\item   If $u, w \in \tau_\prec$ then $\v_u \add \v_w = \v_{u + w}$.
	\item   If $u \in \Trop(I) \subset|\Sigma(I)|$ then $\v_u(x_i) = u_i$.
	\item   For all $u \in \Q^n$, $\v_u$ has kernel $\{0\}$.  
\end{enumerate}
\end{proposition}

\begin{proof}
The fact that $\mathbb{B}_\tau$ is an adapted basis is \cite[Proposition 3.3]{Kaveh-Manon-NOK}.  Also, \cite[Proposition 3.7]{Kaveh-Manon-NOK} implies that $\v_u(x_i) = u_i$ if $w \in \Trop(I)$. The quasivaluations $\v_u, \v_w$, and $\v_{u + w}$ all share the basis $\mathbb{B}_\prec$, so $\v_u \add \v_w = \v_{u + w}$ by \cite[Proposition 4.9]{Kaveh-Manon-NOK}.   For any monomial $\bx^\beta \in \mathbb{B}_\prec$ we have $\v_u(\bx^\beta) = \langle u, \beta \rangle$, and for any $f \in R$ there is some $\bx^\beta \in \mathbb{B}_\prec$ such that $\v_u(f) = \v_u(\bx^\beta)$, so $\v_u$ takes finite values. 
\end{proof}

\section{Proof of main theorems}

The proof of Theorem \ref{main-associatedgraded-quasival} involves a Bertini-type construction and a theorem of Rees, both of which we introduce now.   In all that follows $\k$ is algebraically closed. Let $X = \Proj(R) \subseteq \mathbb{P}(V^*)$, where $V$ is the space of linear forms in $\k[\bx]$, and $\dim(V) = n$.  First we recall the notions of \emph{regular sequence} and \emph{homogeneous system of parameters}, see \cite[Chapter I, Definitions 5.1 and 5.6]{Stanley}.

\begin{definition}[Homogeneous System of Parameters]
Let $R$ be a graded $\k$-algebra with $R_0 = \k$, and Krull dimension $n$. A sequence $y_1, \ldots, y_n \in R$ of homogeneous elements of positive degree is called a homogeneous system of parameters if the Krull dimension of the quotient ring $R/\langle y_1, \ldots, y_n\rangle$ is $0$.
\end{definition}

\begin{definition}[Regular Sequence]
A sequence $y_1, \ldots, y_d$ of elements of a $\k$-algebra $R$ is said to be a regular sequence if $\langle y_1, \ldots, y_d \rangle \neq R$ and the image of $y_i$ in $R/\langle y_1, \ldots, y_{i-1} \rangle$ is non-zero divisor for $1 \leq i \leq d$.
\end{definition}

\noindent
Positively graded Cohen-Macaulay $\k$-algebras are characterized by the fact that every homogeneous system of parameters is a regular sequence \cite[Chapter I, Theorem 5.9]{Stanley}.  Moreover, if $R$ is a graded Cohen-Macaulay $\k$-algebra, and $y_1, \ldots, y_d$ is a regular sequence, then $R/\langle y_1, \ldots, y_d\rangle$ is also Cohen-Macaulay \cite[Theorem 2.1.3 and Exercise 2.1.28]{BrunsHerzog}.

Next we require several variants of the Bertini theorem.   We say a scheme is pure dimensional if all if its irreducible components have the same dimension.  The following is a variant of the Kleiman-Bertini theorem \cite[Corollary 4]{Kleiman}, \cite[Appendix B, 9.2]{Fulton}. 

\begin{theorem}\label{thm-KleimanBertini}
Let $Y, Z \subset \mathbb{P}(V^*)$ be closed subschemes of pure dimensions $d_1$ and $d_2$ respectively, with $d_1 + d_2 \geq n$. Then there is a dense, open subset $U \subseteq \GL_n(\k)$ such that $g^{-1}Y \cap Z$ has pure dimension $d_1 + d_2 - n$ for all $g \in U$. 
\end{theorem}


Next we recall a technical result from \cite[Proposition 1.1]{Benoist}.  We state the result in a slightly different form to emphasize that the properties ``generically reduced" and ``irreducible" can be treated separately.  

\begin{theorem}\label{thm-Benoist}
Let $f: X \to S$ be a proper, flat, finite-type morphism of schemes, and suppose the fibers $X_s$ of $f$ are pure dimensional.  Then the set $S'$ of $s \in S$ such that $X_s$ is generically reduced is open in $S$.  Moreover, the set of points $S''$ of $s \in S$ such that $X_s$ is generically reduced and geometrically irreducible is open in $S$. 
\end{theorem}

We also need the traditional hypersurface form of the Bertini theorem.  The following is a summary of \cite[3.4.10]{joins}, \cite[3.4.12]{joins}, \cite[3.4.13]{joins}, and \cite[3.4.15]{joins}.  See also \cite[Section 1.2]{Benoist}.

\begin{theorem}\label{thm-Bertini}
Let $X \subset \mathbb{P}(V^*)$ be a reduced subscheme with $\dim(X) \geq 1$.  There is a dense, open subset $W \subset \mathbb{A}^n(\k)$ such that for any $g \in W$ the $0$-locus $X_g$ of the form $\ell = \sum g_j x_j$ is reduced.  If $X$ is also irreducible and $\dim(X) > 1$, there is a dense open subset $W \subset \mathbb{A}^n(\k)$ such that the $0$-locus $X_g$ is reduced and irreducible. 
\end{theorem}

We use Theorem \ref{thm-Bertini} to prove the following lemma. 

\begin{lemma}\label{lem-nonempty}
Let $X \subset \mathbb{P}(V^*)$ be as above, and let $\F=\{f_1, \ldots, f_s\} \subset \k[\M_{k \times n}(\k)]$ be a collection of polynomials.  There is a $\k$-matrix $g \in \M_{k \times n}(\k)$ such that $f_r(g) \neq 0$ for all $f_r \in \F$, and:
\begin{enumerate}
\item If $X$ is reduced, then for every collection $A \subset [k]$, $|A| \leq d-1$, the $0$-locus $X_{A, g} \subset X$ of the forms $\ell_i = \sum g_{ij}x_j$ for $i \in A$ is reduced. 
\item If $X$ is reduced and irreducible, then for every collection $A \subset [k]$, $|A| < d-1$, the $0$-locus $X_{A, g} \subset X$ of the forms $\ell_i = \sum g_{ij}x_j$ for $i \in A$ is reduced and irreducible.
\end{enumerate}

\end{lemma}

\begin{proof}
We give the argument for $(2)$, the proof of $(1)$ is similar.  For $k = 1$ we use Theorem \ref{thm-Bertini} to find an open subset $W \subseteq \M_{k \times n}(\k)$ such that $X_{1, g}$ is reduced and irreducible for all $g \in W$.  Any point in the intersection of $W$ and the complements of the hypersurfaces defined by the $f_r \in \F$ then suffices.  Suppose the statement holds up to $k-1$. We write each $f_r \in \F$ as a polynomial in $n$ variables with coefficients  in $\k[\M_{k-1 \times n}(\k)]$, and we let $\F'$ be the set of all such coefficients.   We use the $k-1$ iteration of the statement with $\F'$ to construct $g' \in \M_{k-1 \times n}(\k)$.  Let $\F'' \subset \k[x_1, \ldots, x_n]$ be the set of polynomials obtained from $\F$ by evaluating the coefficients $\F'$ at $g'$.   By construction, $(1)$ and $(2)$ are satisfied for any $A \subset [k]$ with $|A| \leq d-1$ which does not include the index $k$.  Let $A \subset [k]$ include $k$, suppose that $|A| \leq d-1$, and let $B = A \setminus \{k\}$.  The scheme $X_{B, g'}$ is reduced and irreducible, so by Theorem \ref{thm-Bertini}, we can find an open subset $W_A \subset \mathbb A^n(\k)$ such that for any $g_k \in W_A$, the $0$-locus $X_{A, g} \subset X_{B, g'}$ of the form $\ell_k = \sum g_{kj}x_j$  is reduced  if $|A| \leq d-1$, and irreducible if $|A| < d-1$.  Now, any point $g_k$ in the dense, open subset $W \subset \mathbb A^n(\k)$ obtained by taking the intersection of the $W_A$ with $k \in A$ with the complements of the hypersurfaces defined by the members of $\F''$,can be appended to $g'$ to form $g \in \M_{k \times n}(\k)$ which satisfies $(1), (2),$ and $(3)$ above. 
\end{proof}

\begin{proposition}\label{prop-Bertini}
Let $J \subset \k[\bx]$ be a homogeneous ideal such that the positively graded algebra $R = \k[\bx]/J$ is a $d$-dimensional Cohen-Macaulay algebra.  Then there is a dense, open subset $W \subseteq \GL_n(\k)$ such that for any $g \in W$, and any $A \subset [n]$ with $|A| \leq d -1$, the ideal $\I_A = \langle I, x_i \mid i \in A\rangle$ has height $d-|A|$, where $I = g^{-1}\circ J$. Any such set $\{x_i \mid i \in A\}$ defines a regular sequence in $R \cong \k[\bx]/I$.  Moreover, if $J$ is radical then $\I_A$ is radical, and if $J$ is prime then $\I_A$ is prime for all $|A| < d-1$. 
\end{proposition}

\begin{proof}
The ring $R$ is positively graded and Cohen-Macaulay with $R_0 = \k$, so the varieties $\Spec(R)$ and $\Proj(R)$ are Cohen-Macaulay, and pure dimensional.  For $A \subseteq [n]$, let $V_A^* \subset V^*$ be the $0$-locus of $\{x_i \mid i \in A\}$.  By \ref{thm-KleimanBertini} there is a dense, open subset $W'_A \subseteq \GL_n(\k)$ such that $g^{-1}\circ X \cap \mathbb{P}(V_A^*)$ has pure dimension $d + |A^c| - n  = d- |A|$ for all $g \in W'_A$.   If we let $W' = \bigcap_{|A| \leq d-1} W'_A$ and $I = g^{-1}\circ J$ for $g \in W'$, then the image of any set $\{x_i \mid i \in A\}$ in $\k[\bx]/I$ is part of a system of parameters for $R \cong \k[\bx]/I$.  The algebra $R$ is Cohen-Macaulay so the image of any $\{x_i \mid i \in A\}$ in $\k[\bx]/I$ is a regular sequence.  It follows that $\Spec(\k[\bx]/\I_A)$ and $\Proj(\k[\bx]/\I_A)$ are Cohen-Macaulay \cite[Ex 2.1.28]{BrunsHerzog}.

Let $Z = X \times_\k \M_{n \times n}(\k)$, and let $\X_A \subset Z$ be the $0$-locus of the forms $F_j = \sum_{i = 1}^n x_i y_{ji}$ for $j \in A$, where the $y_{ji}$ are polynomial generators of the coordinate ring $\k[\M_{n \times n}(\k)]$. Let $\mathcal{L}$ be the restriction of $\O(1)$ on $\mathbb{P}(V^*)$ to $X$, then the $x_i \in V$ define sections $s_i \in H^0(X, \mathcal{L})$.  Consider the projection $p: \X_A \to X$, and let $U \subset X$ be an affine open subset with coordinate ring $A = \k[U]$.  Without loss of generality we may arrange that $\mathcal{L}\!\!\mid_U = s_1\mathcal{O}_U$ with $s_i = f_is_1$, so that $p^{-1}(U) \subset \X_A$ is the $0$-locus of the equations $y_{j1} = \sum_{i > 1} f_iy_{ji}$ in $U \times_\k \M_{n\times n}(\k)$.  It follows that $p^{-1}(U) \cong U \times_\k \M_{n \times n-1}(\k)$.  As a consequence, we see that if $X$ is reduced or irreducible, $\X_A$ is as well.  Moreover, if $g \in W' \subset \M_{n \times n}(\k)$ is as above, then the fiber $X_{A,g}= p^{-1}(g)$ under the projection $p: \X_A \to \M_{n \times n}(\k)$ is seen to be isomorphic to  $\Proj(\k[\bx]/\I_A)$.  We let $\X_{A, W'}$ be the base-change to $W' \subset \M_{n \times n}(\k)$.  By shrinking $W'$ if necessary, we can ensure that $p: \X_{A, W'} \to W'$ is flat, so $p$ satisfies the criteria of Theorem \ref{thm-Benoist}. Hence, we see that the set of points $W_A \subset W'$ where $X_{A, g}$ is generically reduced (resp. generically reduced and irreducible) is open in $W'$.  Let $W = \bigcap_{|A| \leq d-1} W_A$.  

Now we suppose that $J$ is radical (respectively, prime).  We show that $W$ is non-empty by applying Lemma \ref{lem-nonempty} with $\F = \{f\}$, where $f: \M_{n \times n}(\k) \to \k$ is a polynomial such that $\M_{n \times n }(\k) \setminus V(f) \subseteq W'$.  Note that for every $g \in W$ and $A$ with $|A| \leq d-1$, the corresponding schemes $X_1 = \Proj(\k[\bx]/\I_A)$ and $X_2 = \Spec(\k[\bx]/\I_A)$ are Cohen-Macaulay.  Since $X_1$ is generically reduced (respectively irreducible if $J$ is prime and $|A|< d-1$), $X_2$ is as well.  This implies that $X_2$ satisfies Serre's condition $R_0$.  Moreover, as $X_2$ is Cohen-Macaulay, it satisfies Serre's condition $S_m$ for all $m \geq 0$, so it is $S_1$ and $R_0$ (for definitions of conditions $S_t$ and $R_t$ see for example \cite[Section 1.3, p. 30 and Section 1.4, p. 38]{joins}).  This is equivalent to being reduced. 
\end{proof}

Next we use the following result, due originally to Rees (see \cite[Theorem 2.1]{Rees}).

\begin{proposition}\label{prop-Rees}
Let $\bar{y} = \{y_1, \ldots, y_k\}$ be a regular sequence in an algebra $R$, and consider the $\J$-adic filtration of $R$, where $\J = \langle y_1, \ldots, y_k \rangle$, with associated graded algebra $\gr_\J(R)$. Then we have $\gr_\J(R) \cong (R/\J)[t_1, \ldots, t_k]$.  The isomorphism is defined by sending each $t_i \in (R/\J)[t_1, \ldots, t_k]$ to the image of $y_i$ in $\J/\J^2 \subset \gr_\J(R)$. Moreover, $R/\J$, in the righthand side, is identified with $R/\J$ as the zeroth degree part of $\gr_\J(R)$ in the lefthand side.
\end{proposition}

As before let $R$ be a positively graded algebra. From now on we let $\bar{y} = \{y_1, \ldots, y_n\} \subset R$ be a regular sequence (as in Proposition \ref{prop-Bertini}), and we let $I \subset \k[\bx]$ be the corresponding homogeneous ideal.  We select a subset $A \subset [n]$ with $|A| \leq d-1$, and we let $J_A = \langle y_i \mid i \in A\rangle$.  We let $\ord_A: R \to \bar{\Z}$ be the quasivaluation obtained from the $J_A$-adic filtration of $R$, and $\gr_A(R)$ be the corresponding associated-graded algebra.  Finally, we let $\epsilon_A \in \Q^n$ be the $(0, 1)$-vector with a $1$ for each $i \in A$ and a $0$ for $j \in A^c$. 

\begin{proposition}\label{prop-value}
We have $\epsilon_A = (\ord_A(y_1), \ldots, \ord_A(y_n)) \in \W^n_d$.  In particular, $\epsilon_A \in \Trop(I)$, $\v_{\epsilon_A} = \ord_A$ and $\In_{\epsilon_A}(I) = I_A\k[\bx]$. 
\end{proposition}

\begin{proof}
Observe that $y_i \in J_A \setminus J_A^2$ for $i \in A$ and $y_j \in R \setminus J_A$ for $j \in A^c$.  This and the definition of $\W^n_d$ prove the first claim. Proposition \ref{prop-Rees} implies that the set $\bar{y}$ is a Khovanskii basis of $\ord_A$, so by \cite[Proposition 3.7]{Kaveh-Manon-NOK}, $\ord_A = \v_{\epsilon_A}$ since both quasivaluations take the same values on $\bar{y}$.  Proposition \ref{prop-Rees} also shows that $\In_{\epsilon_A}(I) = I_A\k[\bx]$.  
\end{proof}

Now we fix $A \subset [n]$ with $|A| = d$, and consider the weight vectors $\epsilon_i = \epsilon_{\{i\}}$ for $i \in A$, that is, $\epsilon_i$ is the $(0, 1)$-vector with $1$ at the $i$-th position and $0$ everywhere else.  We show that these all live in a common cone of the Gr\"obner fan. 

\begin{lemma}\label{lem-samecone}
Let $A$ and $i \in A$ be as above, then $\In_{\epsilon_{A\setminus \{i\}}}(\In_{\epsilon_i}(I)) = \In_{\epsilon_A}(I)$. 
\end{lemma}

\begin{proof}
By Proposition \ref{prop-value}, $\In_{\epsilon_i}(I) = I_{i}\k[x_1, \ldots, x_n]$, where $I_{i} \subset \k[x_1, \ldots, \hat{x_i}, \ldots, x_n]$ presents $R/\langle y_i \rangle$.  The algebra $R/\langle y_i \rangle$ is also Cohen-Macaulay, so the images of $\{y_j \mid j \in A \setminus \{i\}\}$ also form a regular sequence. It follows that $\gr_{\epsilon_{A \setminus \{i\}}}(\gr_{\{i\}}(R)) \cong (R/J_A)[t_i \mid i \in A]$, and so $\In_{\epsilon_{A\setminus \{i\}}}(\In_{\epsilon_i}(I)) = \In_{\epsilon_A}(I)$. 
\end{proof}

Now we prove Theorem \ref{main-associatedgraded-quasival} and Corollary \ref{cor-initialideal}. 

\begin{proof}[Proof of Theorem \ref{main-associatedgraded-quasival} and Corollary \ref{cor-initialideal}]
By Proposition \ref{prop-Bertini}, the $y_i$ form a regular sequence in $R$. We observe that $(\deg(y_1), \ldots, \deg(y_n)) = (1, \ldots, 1)$, and that $\In_{w(1, \ldots, 1)}(J) = J$ for any ideal we encounter, as everything is homogeneous.  In particular, $\deg = \v_{(1, \ldots, 1)}$.  For $w \in \CC_A$ we then get $(\min(w) \scale \deg) \add (\ADD_{i \in A} (w_i - \min(w))\scale \ord_i) =$ $\v_{\min(w)(1, \ldots, 1)} \add  (\ADD_{i \in A} \v_{(w_i- \min(w))\epsilon_i})$ $= \v_w$, where the latter equality is a consequence of Proposition \ref{prop-standard-properties}.   
By Lemma \ref{lem-samecone}, the $\epsilon_i$ for $i \in A$ all lie in the same cone of $\Sigma(I)$.  Moreover, for any $w_i > 0$ with $w =\sum_{i \in A} w_i \epsilon_i$, we have $\In_w(I) = \In_{\epsilon_A}(I)$, as $w$ and $\epsilon_A$ are in the relative interior of this cone. It follows that $\gr_w(R) \cong \gr_{\epsilon_A}(R) \cong R/J_A[t_i \mid i \in A]$.  Moreover, $\add$-ing with some multiple of degree does not change this calculation.  
\end{proof}

Now we use Propositions \ref{prop-Bertini} and \ref{prop-noker} to prove Theorem \ref{main-primecones} and Corollary \ref{cor-wellpoised}.

\begin{proof}[Proof of Theorem \ref{main-primecones} and Corollary \ref{cor-wellpoised}]
Let $w \in \CC_A \subset \W_d^n$, then $\gr_w(R) \cong \k[\bx]/\In_w(I) \cong \k[\bx]/\In_{\epsilon_A}(I)$ $\cong R/J_A[t_i \mid i \in A]$.  If $J$ is radical, then $I$ and $I_A$ are radical, so $R/J_A$ is reduced and $\gr_w(R)$ is reduced.  The same reasoning holds if $w \in \W_{d-1}^n$ and $I$ is prime.  Now if $\In_w(I)$ is prime (resp. radical), Proposition \ref{prop-noker} implies that $J$ is prime (resp. radical). 
\end{proof}

To prove Corollary \ref{cor-wellpoised} we show a slightly stronger result.  

\begin{proposition}
Let $R$, $J$ and $I$ be as above, then the following are equivalent:
\begin{enumerate}
\item $\CC_A$ is a prime cone for some $A \subset [n]$ with $|A| = d-1$.
\item $I$ is well-poised. 
\item $I$ is a linear ideal. 
\end{enumerate}
\end{proposition}

\begin{proof}
If $I$ is a linear ideal, then any initial ideal of $I$ is also linear, and therefore prime.  This shows that $(3) \implies (2)$.  Clearly $(2) \implies (1)$, so we show that $(1) \implies (3)$.  Suppose that $\C_A$ were a prime cone for some $|A| = d-1$, then by Theorem \ref{main-associatedgraded-quasival}, $\gr_w(R) \cong R/J_A[t_i \mid i \in A]$ is a domain.  It must be the case that $R/J_A$ is a positively graded domain of dimension $1$ which is generated by its degree $1$ component, so $R/J_A \cong \k[t]$. As a consequence, the monomials in the generators $\{y_i \mid i \in A\} \cup \{t\}$ form a homogeneous $\k$-vector space basis of $R$.  In particular, the Hilbert function of $R$ is the same as the Hilbert function of a polynomial ring on $d$ variables.  Now let $I_1 \subseteq I$ be the ideal generated by the degree $1$ part of $I$. We see that the degree $1$ part of $\k[\bx]/I_1$ has dimension $d$, so $\k[\bx]/I_1$ is isomorphic to a polynomial ring on $d$ variables. This implies that the Hilbert functions of $R = \k[\bx]/I$ and $\k[\bx]/I_1$ agree, so  $I_1 = I$.
\end{proof}

\bibliographystyle{alpha}
\bibliography{Biblio}

\end{document}